\documentclass[11pt]{article}
\usepackage[francais,english]{babel}
\usepackage[latin1]{inputenc}
\usepackage[T1]{fontenc}
\usepackage{amsmath,amssymb}
\usepackage{graphicx}

\def\acknowledgement{\medskip {\bf {Acknowledgements}: }}

\let\f=\frac
\def\R{\mathbb{R}}
\def\E{\mathbb{E}}
\def\N{\mathbb{N}}
\def\Id{\,{\rm Id}}
\def\div{{\rm div}}
\def\tr{{\rm tr}}
\def\XX{\mathbf{X}}
\def\BB{\mathbf{B}}
\def\KK{\mathbf{K}}
\def\Pe{{\rm Pe}}
\def\var{\text{Var}}

\newtheorem{prop}{Proposition}[section]

\newtheorem{remark}{Remark}

\title{Periodic long-time behaviour for an approximate model of nematic polymers}

\author{Lingbing He$^1$, Claude Le Bris$^2$, Tony Leli\`evre$^2$\\
{\footnotesize 1- Department of Mathematical Sciences,}\\ 
{\footnotesize Tsinghua University, Beijing,}\\
{\footnotesize  China, 100084.}\\
{\footnotesize\tt lbhe@math.tsinghua.edu.cn}\\
\footnotesize{2- CERMICS, \'Ecole Nationale des Ponts et
Chauss\'ees,}\\ 
{\footnotesize 6 \& 8, avenue Blaise Pascal,}\\
{\footnotesize  77455 Marne-La-Vall\'ee
  Cedex 2,}\\ 
{\footnotesize and}\\
{\footnotesize  INRIA Rocquencourt, MICMAC project-team,}\\
{\footnotesize  Domaine de Voluceau, B.P. 105,}\\
{\footnotesize  78153 Le Chesnay Cedex, FRANCE.}\\
{\footnotesize\tt \{lebris,lelievre\}@cermics.enpc.fr}
}

\begin{document}

\maketitle

\begin{abstract}
We study the long-time behaviour of a nonlinear Fokker-Planck equation, which models the evolution of rigid polymers in a given flow, after a closure approximation. The aim of this work is twofold: first, we propose a microscopic derivation of the classical Doi closure, at the level of the kinetic equation ; second, we prove the convergence of the solution to the Fokker-Planck equation to periodic solutions in the long-time limit.
\end{abstract}

\section{Introduction}

In a previous work~\cite{jourdain-le-bris-lelievre-otto-06}, two of us have studied the long time
behaviour of some flows of infinitely dilute flexible polymers. It has been proved that for appropriate boundary data, and
provided the solution  is assumed regular, the solution returns to
equilibrium in a long time, whatever the initial condition. The
mathematical ingredient for the proof is the appropriate use of
Log-Sobolev inequalities and entropy methods {\it \`a la
  Desvillettes-Villani}, applied to the Fokker-Planck equation
derived from statistical mechanics and kinetic theory considerations and
modelling the evolution of the polymer chains.  It turns out that flows of \emph{rigid} polymers
exhibit equally interesting and, actually, much more varied properties in the long time, in particular stationary periodic-in-time motions. Several behaviours have been experimentally observed. Numerical simulations
confirm the ability of the models employed to reproduce such
behaviours. We refer for example to~\cite{constantin-kevrekidis-titi-04_a,constantin-kevrekidis-titi-04_b,zhang-zhang-04_doi} for previous mathematical studies. Those behaviours are
traditionally filed in different categories with an appropriate
terminology. One speaks of flows exhibiting \emph{kayaking}, \emph{tumbling},
etc. Mathematically, one underlying question that might be considered is
to investigate whether the
solution to the Fokker-Planck equation ruling the evolution of the
microstructure (here, typically rigid rods) converges in the long term to a
periodic-in-time solution. This is in sharp contrast to the case of
flexible polymers considered in~\cite{jourdain-le-bris-lelievre-otto-06} where the long time limit of
the solution to the Fokker-Planck equation is
a \emph{steady}, that is \emph{time-independent} density. The
mathematical ingredients mentioned above (Log-Sobolev inequalities and
entropy methods) are again useful, but their use is more delicate, as
will be seen below.

The purpose of this article is to consider a simple setting, where the
long term behaviour of the evolution of the microstructures can be
proven to indeed be periodic. 

To start with, we consider a commonly used model, namely the rigid rod model with a Maier-Saupe potential. It has been observed numerically~\cite{zhang-zhang-04_doi} that, for this specific
model,
the flow is, in the long time, periodic-in-time. It is therefore an adequate
setting to consider with the view to proving that, in the long time, the
solution to the Fokker-Planck equation converges to a periodic-in-time
solution. This Fokker-Planck equation formally writes
\begin{equation}
  \label{eq:FP-11}
  {{\partial\Psi}\over{\partial t}}= L(\Psi)\,\Psi
\end{equation}
where $L(\Psi)$ is a nonlinear nonlocal partial differential operator, essentially
parabolic (see the weak formulation~\eqref{eq:FP1orig} below) and $\Psi$ is the probability
density function describing the state of the microstructure.  
We are unfortunately
unable to prove mathematically the expected long term behaviour of the rigid rod model with a Maier-Saupe potential. Note
that this is indeed a particularly challenging issue to make a period appear in
an equation of the form~\eqref{eq:FP-11} which does not explicitly contain any periodic
function to start with! In some sense, we would need a Poincar\'e
Bendixson type theorem for an infinite dimensional system. This is
beyond our reach.

So we proceed somewhat
differently. In Sections~\ref{sec:closed} and~\ref{sec:derivation}, we first derive an \emph{approximation} of the model, that gives rise to a \emph{closed} evolution
equation for the so-called \emph{conformation tensor}. This equation agrees with the equation obtained when a
classical Doi-type closure is performed on the original model. In passing, we motivate in Section~\ref{sec:proj_moy} our particular choice of
approximate model.  We finally prove, in Section~\ref{sec:convergence}, that the solution to this equation
does behave as expected in the long time: it becomes periodic-in-time. Our proof falls in essentially
two steps. We first show, in Section~\ref{ssec:existence}, that the conformation tensor $\displaystyle \int x \otimes x \, \Psi$
calculated from the solution $\Psi$ to our approximate model, and which  satisfies the
\emph{closed}   evolution equation~\eqref{eq:doi_closure} (this is the
whole point of the closure approximation), becomes periodic-in-time in the long
term. We next use this result in our final Section~\ref{ssec:convergence} 
to conclude our study on the convergence of the density $\Psi$
itself. Because our fundamental tool (in the course of Section~\ref{ssec:existence}) is the Poincar\'e-Bendixson
Theorem (following the work~\cite{lee-forest-zhou-06}), our main result is unfortunately restricted to the
two-dimensional setting. Other intermediate results we prove however
hold whatever the dimension (and they have been proved and stated
so). Another technical limitation lies in the fact we exploit, for our
proof, the specific {\em explicit} expression of the time-periodic solution we
establish the existence of, and which attracts all solutions in the
long-time. More generality in the technique of proof would be highly
desirable but is out of our reach to date.  

In summary, the main contributions of this work are of two different
types in nature: from a modelling viewpoint, we propose a microscopic derivation of the quadratic Doi closure, using a stochastic dynamics with constraints on an average quantity (see Proposition~\ref{prop:doi}) ; from a mathematical viewpoint, we analyze the longtime behaviour of the solution to a nonlinear Fokker-Planck equation which converges to a periodic in time function (see Proposition~\ref{prop:CV}).

Our study can be considered in the vein of several previous studies such
as~\cite{dolbeault-kinderlehrer-kowalczyk-02,bartier-dolbeault-illner-kowalczyk-07}. It is the authors' wish that the quite limited study performed here will
be yet another incentive for mathematicians to consider the question
of long term convergence to non steady stationary states for solutions
to kinetic equations of the type~\eqref{eq:FP-11}.  
We reiterate that, in our opinion, the issue of proving, in some
particular settings and under
appropriate conditions, that solutions to nonlinear
Fokker-Planck equations of the type~\eqref{eq:FP-11} converge in
the long time to periodic-in-time solutions is an interesting, unsolved
mathematical issue.

\section{The Original model and the Doi closure}\label{sec:doi}

\subsection{The Maier-Saupe model for rigid polymers}
\label{sec:original}
Using the notation of~\cite{lee-forest-zhou-06}, we consider, in
Stratonovich form, the following stochastic adimensionalized model for a
rigid
polymer with the Maier-Saupe potential:
\begin{equation}
\label{eq:Xt-strat}
  dX_t = P(X_t)( \kappa X_t + 4N \E(X_t \otimes X_t) X_t ) \, dt + \sqrt{2} P(X_t) \circ dB_t,
\end{equation}
where $B_t$ denotes a $d$-dimensional Brownian motion, $N$ is a dimensionless concentration parameter and  the matrix $\kappa \in \R^{d \times d}$  is related to the velocity field of the ambient flow, which is assumed to be homogeneous (so that transport terms in~\eqref{eq:Xt-strat} are omitted). The purpose of the projector operator
\begin{equation}\label{eq:P}
P(X) = \Id - \frac{X \otimes X}{\|X\|^2}
\end{equation}
is to ensure the preservation of the rigid polymer (nematic crystalline
polymer) norm $\|X_t\|$ in time:
$$d \|X_t\|^2 = 0.$$
Here and in the following, we use the tensor product notation: for two vectors $u$ and $v$ in $\R^d$, $u \otimes v$ is the $\R^{d \times d}$ matrix whose $(i,j)$-entry is $u_i v_j$. Notice that $P(X)$ is a symmetric matrix such that $P(X)P(X)=P(X)$.

The right-hand side of Equation~\eqref{eq:Xt-strat} contains three terms
which model different phenomena. The first term models the reorientation
due to the velocity gradient $\kappa$ of the fluid. The second,
non-linear term contains the force associated to the Maier-Saupe potential, which describes the effective interaction of a rigid rod with the other rods. The rightmost term is a rotational
diffusion term.

In the following, we will in particular consider the two-dimensional case ($d=2$) and a simple shear flow, for which
\begin{equation}\label{eq:kappa}
  \kappa=\frac{\Pe}{2}
\left[
\begin{array}{cc}
0 & a+1 \\
a-1 & 0
\end{array}
\right]
\end{equation}
where $\Pe$ (the Péclet number) and $a$ (a molecular shape parameter) are two adimensional parameters, see~\cite{lee-forest-zhou-06}. But for the moment being, we consider~\eqref{eq:Xt-strat} in full generality.

Equation
\eqref{eq:Xt-strat} equivalently writes, using It\^o's integration rule:
\begin{equation}
\label{eq:Xt}
dX_t = P(X_t)( \kappa X_t + 4N\E(X_t \otimes X_t) X_t ) \, dt + \sqrt{2} P(X_t) dB_t - (d-1) \frac{X_t}{\|X_t\|^2} \, dt,
\end{equation}
where $d$ is the dimension of the ambient space. Obtaining~\eqref{eq:Xt} from~\eqref{eq:Xt-strat} is straightforward using (with implied summation over repeated indices)
\begin{eqnarray*}  \sqrt{2} P_{i,j}(X_t) \circ dB^j_t=\sqrt{2} P_{i,j}(X_t) dB^j_t+\f12
  (\sqrt{2} P_{j,k}(X_t)\partial_j)\cdot \sqrt{2} P_{i,k}(X_t)dt, \end{eqnarray*}
and
 \begin{align*}   P_{j,k}(x)\partial_j P_{i,k}(x)
&=(\delta_{j,k}- x_j x_k / |x|^2) \, \partial_j(\delta_{i,k}- x_i x_k / |x|^2)\\
&=(1-d) x_i / |x|^2,
\end{align*}
where $\delta_{i,j}$ is the Kronecker symbol.

The Fokker-Planck equation associated to~\eqref{eq:Xt} (or
equivalently to~\eqref{eq:Xt-strat}) writes, in weak form: for any smooth test function $\varphi: \R^d \to \R$,
\begin{equation}
  \label{eq:FP1orig}
\begin{aligned}
\frac{d}{dt} \int \varphi(x) \mu_t(dx) &= \int P(x)
  \left(\kappa x + 4N \left(\int x \otimes x \, \mu_t(dx)\right) x \right) \cdot \nabla \varphi (x)  \mu_t(dx)  \\
&\quad+ \int (P_{j,k}(x) \partial_j) \left( P_{i,k}(x) \partial_i \varphi \right)\mu_t(dx),
\end{aligned}
\end{equation}
with again implied summation over repeated indices and where for any time $t \ge 0$, $\mu_t(dx)$ is the law of $X_t$ (with support on a sphere).

 \subsection{The Doi closure to derive a closed equation for $\E(X_t \otimes X_t)$ }
 \label{sec:closed}

As announced in the introduction, we now recall how to derive from
 \eqref{eq:Xt} a \emph{closed} evolution equation on the conformation tensor $$M(t) = \E(X_t \otimes X_t),$$
using the standard quadratic Doi closure approximation~\cite{doi-81}. It is the aim of Section~\ref{sec:proj_moy} below to provide a microscopic justification of this closure.

Using elementary It\^o differential calculus (and the properties $P=P^T$ and $P^2=P$), we compute:
\begin{align*}
d (X_t \otimes X_t)&=
dX_t \otimes X_t + X_t \otimes dX_t + 2 P P^T(X_t) dt\\
&= \left(
 P(X_t) \kappa X_t \otimes X_t +  X_t \otimes X_t \kappa^T P(X_t) \right) \, dt\\
&\quad + 4N \left( P(X_t) \E(X_t \otimes X_t) X_t \otimes X_t + X_t \otimes X_t \E(X_t \otimes X_t) P(X_t)  \right) \, dt \\
&\quad - 2 (d-1) \frac{X_t \otimes X_t}{\|X_t\|^2} \, dt  + 2 P(X_t) \,  dt \\
& \quad + \sqrt{2} P(X_t) dB_t \otimes X_t + \sqrt{2} X_t \otimes dB_t P(X_t)\\
&= \left(
 P(X_t) \kappa X_t \otimes X_t +  X_t \otimes X_t \kappa^T P(X_t) \right) \, dt\\
&\quad + 4N \left( P(X_t) \E(X_t \otimes X_t) X_t \otimes X_t + X_t \otimes X_t \E(X_t \otimes X_t) P(X_t)  \right) \, dt \\
&\quad +2 \left( \Id - d \frac{X_t \otimes X_t}{\|X_t\|^2} \right)\, dt \\
& \quad + \text{loc.mart.}
\end{align*}
where ${\it loc. mart.}$ denotes a
  local martingale that we do not need to make precise for the rest of
  our argument. 
Taking the trace of the previous equation
 and using that   $\tr(AB)=\tr(BA)$ and  $(X\otimes X)\, P(X) = P(X)\,( X
 \otimes X) = 0$, we check the preservation of the  norm
$$d \|X_t\|^2=0$$
as was announced earlier. We henceforth set $\|X_t\|=L$.

Taking now the expectation, we obtain:
\begin{equation}\label{eq:M}
\begin{aligned}
\frac{d M}{dt}&= \kappa M + M \kappa^T - \frac{2}{L^2} \E\left(  \kappa : X_t \otimes X_t \, X_t \otimes X_t\right)\\
&\quad+ 4N \left( 2 M^2 - \frac{2}{L^2} \E \left(  M: X_t \otimes X_t \, X_t \otimes X_t \right) \right) \\
& \quad +  2 \Id - \frac{2 d}{L^2} M,
\end{aligned}
\end{equation}
where we here introduced the Frobenius inner product: for two $\R^{d\times d}$ matrices $A$ and $B$, $A:B = \tr (AB^T)= \sum_{i,j=1}^d A_{i,j} B_{i,j}$.

At this stage, we use the so-called \emph{quadratic Doi closure}~\cite{doi-81} that consists in performing
the following approximation: for any deterministic matrix $K$,
$$\E \left(K X_t \otimes X_t \, X_t \otimes X_t\right)=K:\E(X_t \otimes X_t) \, \E (X_t \otimes X_t).$$
The following \emph{closed} nonlinear first order differential equation
that rules the evolution of~$M$ in time is thereby obtained (using the fact that $\tr(M)=L^2$):
\begin{equation}
\label{eq:doi_closure}
\begin{aligned}
\frac{d M}{dt}&= \kappa M + M \kappa^T - \frac{2}{\tr(M)}
\kappa : M  \, M\\
&\quad+ 4N \left( 2 M^2 - \frac{2}{\tr(M)} M:M \, M \right)\\
& \quad +  2 \Id - \frac{2d}{\tr(M)} M.
\end{aligned}
\end{equation}
Note that, at this stage, the above equation is formal since we do not
know that $\tr(M)$ does not vanish. It is a consequence
of the following proposition that this is not the case.

\begin{prop}
\label{prop:sym/trace}
Assume that we supply equation~\eqref{eq:doi_closure} with an initial
condition $M(0)$ that is a symmetric matrix and that satisfies
$\tr(M(0))=L^2>0$. Then there exists a unique solution $M(t)$ to~\eqref{eq:doi_closure}. Moreover, this solution
remains symmetric for all times and with constant trace: $\tr (M(t))=L^2$.
\end{prop}
\begin{proof}
We consider a time interval on which equation~\eqref{eq:doi_closure} is
well posed. Such a time interval exists by a standard application of the
Cauchy-Lipschitz Theorem. Indeed,  the right-hand
side of~\eqref{eq:doi_closure} is a rational function in the
coefficients of $M$. Momentarily, this time interval may be bounded, but we will soon
see that it is in fact infinite. 

The  adjoint matrix $M ^T(t)$ of $M(t)$ then satisfies the following equation:
\begin{align*} 
\frac{d M^T}{dt}&=   \kappa  M ^T+ M^T \kappa^T - \frac{2}{\tr(M)}
\kappa : M  \, M^T  \\
&\quad+  4N \left( 2(M^T)^2 -\frac{2}{\tr(M)} M:M \, M^T \right)  \\
& \quad +  2\Id - \frac{2d}{\tr(M)} M^T. \end{align*}

It follows that, when $M(t)$ solves~\eqref{eq:doi_closure},  both $M(t)$ and $M^T(t)$ are solutions to the first order
evolution equation
\begin{align*} \frac{d B}{dt}&=    \kappa  B+ B \kappa^T - \frac{2}{\tr(M)}
\kappa : M  \, B  \\
&\quad+  2\,B^2 -\frac{2}{\tr(M)} M:M \, B   \\
& \quad +  2 \Id - \frac{2d}{\tr(M)} B, \end{align*}
and that this holds for the \emph{same} initial condition $M(0)$ since the
latter is symmetric. Now, the right hand side of this differential equation is a second order
polynomial in $B$, with coefficients that are obviously continuous in time (in
turn because $M(t)$ solves~\eqref{eq:doi_closure}). It
follows that the Cauchy-Lipschitz theorem holds for this equation, and thus
that the solution is unique for a given initial condition. This proves
that $M(t)=M^T(t)$ for all times in the considered time interval.

We now take  the trace of \eqref{eq:doi_closure} and we easily check that
$$\frac{d} { dt} \,\tr(M)= 0,$$
using the fact that $M$ is now known to be symmetric. The trace of the solution is thus preserved in time and this in
particular shows that the solution to~\eqref{eq:doi_closure} is defined for any time (for any initial condition
with non zero trace).
We finally notice, as this will be useful below that, since $\tr(M)|_{t=0}=L^2$, \eqref{eq:doi_closure} also writes
\begin{equation}
\label{eq:doi_closure_prime}
\begin{aligned}
\frac{d M}{dt}&= \kappa M + M \kappa^T - \frac{2}{L^2}
\kappa : M  \, M\\
&\quad+ 2 M^2 - \frac{2}{L^2} M:M \, M\\
& \quad +  2 \Id - \frac{2d}{L^2} M.
\end{aligned}
\end{equation}
\end{proof}

 \begin{remark}
   A natural question is  whether Equation~\eqref{eq:doi_closure}
  also preserves positiveness (in the sense of symmetric
   matrices). This property will be a consequence of a rewriting of the
   solution to~\eqref{eq:doi_closure} as $M(t)=\E(X_t \otimes X_t)$ for
   $X_t$ solution to a modified stochastic differential equation, see
   Section~\ref{sec:derivation} below. We are unable to prove this
   preservation otherwise.
 \end{remark}

\subsection{A rewriting of the equations}\label{sec:forest}
It is enlightening to compare our
equation~\eqref{eq:M} with the equation~(10)  in the article~\cite{lee-forest-zhou-06}  by
G. Forest and collaborators. For this purpose we recall the three
dimensionless numbers used by these authors:
the molecular shape parameter~$a$, the P\'eclet number~$\Pe$  and the
dimensionless concentration number~$N$. In this context, the dimension is $d=2$, the length is $L=1$ and $\kappa$ is~\eqref{eq:kappa} and thus writes:
$$\kappa=\Pe \left( \Omega + a D \right)$$
where $\Omega=\frac{1}{2}\left[\begin{array}{cc} 0 & 1 \\ -1  &
    0 \end{array}\right]$ and $D=\frac{1}{2}\left[\begin{array}{cc} 0 &
    1 \\ 1  & 0 \end{array}\right]$.

The equation~\eqref{eq:M} then reads (using the fact that $\Omega$ is skew-symmetric):
\begin{equation}\label{eq:M_forest}
\begin{aligned}
\frac{d M}{dt}&= \Pe \left( \Omega M - M \Omega + a (D M + M D ) - 2 a \E\left(  D : X_t \otimes X_t \, X_t \otimes X_t\right) \right)\\
&\quad+ 8N \left(  M^2 -  \E \left(  M: X_t \otimes X_t \, X_t \otimes X_t \right) \right) \\
& \quad +  4 (\Id/2 -  M).
\end{aligned}
\end{equation}
To agree with the notation of~\cite{lee-forest-zhou-06}, we introduce
$Q=M-\Id/2$, the traceless part of $M$. Equation~\eqref{eq:M_forest} then rewrites:
\begin{equation}\label{eq:Q_forest}
\begin{aligned}
\frac{d Q}{dt}&= \Pe \left( \Omega Q - Q \Omega + a (D Q + Q D + D) - 2 a \E\left(  D : X_t \otimes X_t \, X_t \otimes X_t\right) \right)\\
&\quad+ 8N \left(  Q^2 + Q/2  -  \E \left(  Q : X_t \otimes X_t \, X_t \otimes X_t \right) \right) \\
& \quad -  4 Q.
\end{aligned}
\end{equation}
We observe
  that the above equation agrees with the Equation (10) in
  \cite{lee-forest-zhou-06} up to multiplicative constants that do not affect our conclusions. Using now the Doi closure approximation, we finally get
\begin{equation}
\label{eq:Q}
\begin{aligned}
{{d Q}\over{d t}}&=\Pe\, \left[\Omega Q-Q\Omega+a(DQ+QD + D)-2aD: Q\left(Q+\f12 \Id\right) \right]\\
&\quad-4\,\left[Q-2N\left(Q+\f12 \Id\right)Q+2NQ:Q  \left(Q+\f12 \Id \right) \right].
\end{aligned}
 \end{equation}

\section{Derivation of an evolution equation for $X_t$ that yields our closed equation on
  $M$}
\label{sec:derivation}

We have derived in the previous section a closed equation~\eqref{eq:doi_closure}  on
$M$, using a classical closure technique ({\it \`a la Doi}) on the
original dynamics~\eqref{eq:Xt} on $X_t$. The question now naturally arises
to know whether it is possible to modify the original stochastic dynamics
\eqref{eq:Xt} \emph{itself} so that $M(t)=\E (X_t \otimes X_t) $ calculated
from $X_t$ solution to this modified dynamics is  solution
to~\eqref{eq:doi_closure}.

\subsection{Two possible closures on the stochastic differential equation}

To begin with, using the fact that for any vector $x \in \R^d$ and matrix $A \in \R^{d \times d}$, $(x\otimes x)\,A\,x=A:(x\otimes x)\, x$, we write~\eqref{eq:Xt} under the form:
\begin{align*}
dX_t &=( \kappa X_t + 4N \E(X_t \otimes X_t) X_t ) \, dt \\
&\quad-  \frac{1}{\|X_t\|^2} \left(\kappa: (X_t \otimes X_t) X_t + 4N \E(X_t \otimes X_t) : (X_t \otimes X_t) \, X_t \right) \, dt\\
&\quad + \sqrt{2} P(X_t) dB_t - (d-1) \frac{X_t}{\|X_t\|^2} \, dt.
\end{align*}
We next modify this equation as follows:
\begin{equation}
\label{eq:Xt_modif}
\begin{aligned}
dX_t &= \left( \kappa X_t   + 4N \E(X_t \otimes X_t) X_t\right) \, dt \\
& \quad - \frac{1}{\E(\|X_t\|^2)} \left( \kappa: \E( X_t \otimes X_t)  \, X_t + 4N \E( X_t \otimes X_t ): \E(X_t \otimes X_t) \, X_t \right) \,dt \\
& \quad + \sqrt{2} R_t\, dB_t- \lambda\, \frac{X_t}{\E(\|X_t\|^2)} \, dt,
\end{aligned}
\end{equation}
where the pair $(R_t,\lambda)$ is yet to be determined so that $M(t)=\E (X_t \otimes X_t) $ calculated from $X_t$ solution to~\eqref{eq:Xt_modif}
is indeed solution to~\eqref{eq:doi_closure}. Here, $R_t \in \R^{d \times d}$ is an adapted stochastic process, and $\lambda \in\R$ is a deterministic constant.

As above, an elementary  It\^o calculation yields
 \begin{align*}
d &(X_t \otimes X_t)=
dX_t \otimes X_t + X_t \otimes dX_t + 2 R_t R_t^T dt\\
&=\Big( \kappa X_t \otimes X_t  + X_t \otimes X_t \kappa^T + 8N \E(X_t \otimes X_t)  X_t \otimes X_t\Big) \, dt\\
& \quad - \frac{1}{\E(\|X_t\|^2)} \left( 2 \kappa: \E( X_t \otimes X_t)  \, X_t \otimes X_t  + 8\,N \E( X_t \otimes X_t ): \E(X_t \otimes X_t) \, X_t \otimes X_t \right) \, dt\\
& \quad -2\lambda \frac{X_t\otimes X_t}{\E(\|X_t\|^2)} \, dt + 2 R_tR_t^T \, dt \\
& \quad + \text{loc.mart.}
\end{align*}
Taking the expectation, we obtain
\begin{align*}
\frac{dM}{dt}&=\left( \kappa M  + M \kappa^T + 8 N M^2\right)\\
& \quad - \frac{1}{\tr(M)} \left( 2 \kappa: M  \, M + 8N M:M \, M \right) \\
& \quad -2\lambda \frac{M}{\tr(M)}  + 2 \E(R_tR_t^T) .
\end{align*}
This equation is then equivalent to equation~\eqref{eq:doi_closure} if
and only if
$$- \lambda\frac{M}{\tr(M)} + \E(R_t R_t^T) = \Id - d \frac{M}{\tr(M)}.$$

Obviously, the simplest possible choice for $(R_t,\lambda)$ is to set
\begin{equation}
\label{eq:choix1}
R_t=\Id \text{ and } \lambda=d.
\end{equation}
It follows that the diffusion term in  the associated Fokker-Planck equation is simply a Laplacian. Let us write the non-linear Fokker-Planck equation we thus obtain:
\begin{equation}\label{eq:FP1}
\begin{aligned}
\frac{\partial \psi}{\partial t} &=
\div \left( \left( - \kappa x + \frac{\kappa:M[\psi]}{\tr(M[\psi])} x \right) \psi \right)\\
 &\quad+ 4N 
\div \left( \left( - M[\psi] x + \frac{M[\psi]:M[\psi]}{\tr(M[\psi])} x \right) \psi \right)\\
& \quad + \Delta \psi + d \,\div \left( \frac{x}{\tr(M[\psi])} \psi \right)
\end{aligned}
\end{equation}
where 
\begin{equation}\label{eq:Mpsi}
M[\psi(t,\cdot)]=\int_{\R^2} x \otimes x \, \psi(t,x) \, dx.
\end{equation}
We will see in the next section that this particular choice~\eqref{eq:choix1}
of pair~$(R_t,\lambda)$ may be motivated by modeling considerations. This
turns out to be the choice we advocate.

An alternate convenient pair (among many possible choices)
would be to set
\begin{equation}\label{eq:choix2}
R_tR_t^T= \Id - \frac{M}{\tr(M)} \text{ and } \lambda=(d-1).
\end{equation}
Using the fact that $M \le \tr(M) \Id$ (in the sense of symmetric matrices), the existence of such a matrix $R_t$ follows from Cholesky factorization, for example. Then, the associated Fokker-Planck equation would write
\begin{equation}\label{eq:FP2}
\begin{aligned}
\frac{\partial \psi}{\partial t} &=
\div \left( \left( - \kappa x + \frac{\kappa:M[\psi]}{\tr(M[\psi])} x \right) \psi \right)\\
 &\quad+ 4N
\div \left( \left( - M[\psi] x + \frac{M[\psi]:M[\psi]}{\tr(M[\psi])} x \right) \psi \right)\\
& \quad+ \Delta \psi - \frac{M[\psi]}{\tr(M[\psi])} : \nabla^2 \psi + (d-1) \div \left( \frac{x}{\tr(M[\psi])} \psi \right),
\end{aligned}
\end{equation}
where $M[\psi]$ is defined by~\eqref{eq:Mpsi}. We have not been able to motivate the alternate choice~\eqref{eq:choix2}
as convincingly as the choice~\eqref{eq:choix1}, and will
show that our preferred choice~\eqref{eq:choix1} enjoys several agreeable
properties. We will therefore henceforth adopt~\eqref{eq:choix1}.

\subsection{A possible justification of our approximation~\eqref{eq:Xt_modif}}
\label{sec:proj_moy}

To derive an appropriate approximation of Equation~\eqref{eq:Xt}, we now follow a different
path. Since~\eqref{eq:Xt} is the projection  on the manifold defined by the
constraint "$\|X_t\|^2$ constant" of an original dynamics 
\begin{equation}
  \label{eq:original-dynamics}
dX_t = \left( \kappa X_t + 4N \E(X_t \otimes X_t) X_t \right) \, dt + \sqrt{2} dB_t
\end{equation}
visiting the whole space $\R^d$, we may consider an approximation
of~\eqref{eq:Xt} as a projection of the same original dynamics on a
"manifold" defined by the constraint "$\E\|X_t\|^2$ constant". 
The difficulty is that giving a mathematical meaning to the latter
constraint is not straightforward. The aim of this section is to give a proper meaning to this projection, and to identify the projected dynamics with the dynamics~\eqref{eq:Xt_modif}--\eqref{eq:choix1} that leads to the Doi closure. To keep our exposition simple, we omit the nonlinear term in the drift term of~\eqref{eq:original-dynamics} (namely $N =0$). The reasoning below generalizes to the full drift.

The approach we propose is to consider $I \ge 1$ replicas 
$$dX^i_t = \kappa X^i_t \, dt +  \sqrt{2} dB^i_t,$$
(for $1\leq i\leq I$) 
of the dynamics~\eqref{eq:original-dynamics} and project the system
obtained on the manifold
$$\frac{1}{I} \sum_{i=1}^I \|X^i_t\|^2=L^2.$$
We thus impose that the empirical average is constant, and we are interested in the limit $I \to \infty$.
Of course, the $d$-dimensional Brownian motions $B^i_t$ are assumed to be independent.
The projection is performed using the
\emph{D'Alembert Principle}. Indeed, the constraining force does not bring or subtract energy from the system: it is directed orthogonally to the submanifold on which the constrained system evolves. More precisely, denoting by $\XX_t=(X^1_t, \ldots , X^I_t)
\in \R^{d I}$, the projected dynamics writes (see for example~\cite{ciccotti-lelievre-vanden-einjden-08,le-bris-lelievre-vanden-eijnden-08}):
\begin{equation}\label{eq:sde_proj}
d \XX_t = P(\XX_t) \KK \XX_t\, dt + \sqrt{2} P(\XX_t) d \BB_t - (d I-1) \frac{\XX_t}{\|\XX_t\|^2} dt,
\end{equation}
where $\KK$ is the $dI \times dI$ block diagonal matrix composed of the
blocks $\kappa$ of size $d \times d$, and $P(\XX)$ is still defined by~\eqref{eq:P}, with $\XX \in \R^{dI}$. We fix 
\begin{equation}\label{eq:IC1}
\frac{1}{I} \sum_{i=1}^I\|X^i_0\|^2=L^2,
\end{equation}
at initial time and this quantity is by construction preserved in time. We also
assume that the random variables $X^i_0$ are identically distributed so that, from~\eqref{eq:IC1},
\begin{equation}\label{eq:IC2}
\E(\|X^1_0\|^2)=L^2.
\end{equation}
As mentioned above, Equation~\eqref{eq:Xt} is recovered using this projection procedure with only one replica: $I=1$. Here, we consider the limit $I \to \infty$.

We now pick the first component $X^1_t \in \R^d$ of our vector $\XX_t$ and
consider its evolution equation
$$dX^1_t = \kappa X^1_t \, dt + \sqrt{2} d B^1_t - X^1_t \frac{\XX_t \cdot \KK \XX_t}{\|\XX_t\|^2} \, dt - X^1_t \frac{\XX_t \cdot d\BB_t}{\|\XX_t\|^2} - (d I - 1) \frac{X^1_t}{\|\XX_t\|^2} dt.$$
Since $\|\XX_t\|^2=I \, L^2$, this also writes
$$dX^1_t = \kappa X^1_t \, dt + \sqrt{2} d B^1_t - \frac{X^1_t}{L^2} \frac{1}{I} \sum_{i=1}^I X^i_t \cdot \kappa X^i_t \, dt - \frac{X^1_t}{L} \frac{1}{\sqrt{I}} \frac{\XX_t \cdot d\BB_t}{\|\XX_t\|} - \left( d - \frac{ 1}{I}\right) \frac{X^1_t}{L^2} dt.$$
In the limit~$I \to \infty$, we formally obtain that $X^1_t$ converges to $Y_t$ solution to
\begin{equation}\label{eq:Y}
dY_t = \kappa Y_t \, dt + \sqrt{2} d B^1_t  - \frac{Y_t}{L^2} \, \E(Y_t \cdot \kappa Y_t) \, dt - d \,\frac{Y_t}{L^2} dt.
\end{equation}
This limit may be rigorously justified as follows.
\begin{prop}\label{prop:doi}
Let $\XX^I_t=(X^{1,I}_t, \ldots,X^{I,I}_t) \in \R^{dI}$ be a solution to~\eqref{eq:sde_proj} (we here explicitly indicate in the superscript the dependence on the number of replicas $I$) and $Y_t$ solution to~\eqref{eq:Y}. The initial condition $Y_0$ is assumed to satisfy
$$\E(\|Y_0\|^2)=L^2$$
and $\E(\|Y_0\|^8)< \infty$. Consider i.i.d. copies $Y^i_0$ of $Y_0$ and define the initial condition
$$X^{i,I}_0=L\,Y^i_0\,\left(\frac{1}{I} \sum_{i=1}^I \|Y^i_0 \|^2\right)^{-1/2}$$
so that $(X^{i,I}_0)_{1 \le i \le I}$ are identically distributed random variables satisfying~\eqref{eq:IC1}. Then, for any positive time $T > 0$, there exists $C >0$ such that, for all positive $I \in \N$,
$$\E\left( \sup_{0 \le t \le T} \| X^{1,I}_t - Y_t \|^2 \right) \le \frac{C}{I}.$$
\end{prop}
The proof is provided in the appendix. Of course, the convergence result holds for any component $X^i_t$ of
the vector $\XX_t$, and by standard results in propagation of chaos~\cite{sznitman-91}, we actually have that any subset of components $(X^{i_1}_t, \ldots X^{i_k}_t)$ converges (in the limit $I \to \infty$) to $(Y^{i_1}_t, \ldots, Y^{i_k}_t)$, where the processes $(Y^i_t)$ are independent copies of $Y_t$ solution to~\eqref{eq:Y}.

Notice that since $\E(\|Y_t\|^2)=L^2$, we may therefore equally well write the dynamics
on~$Y_t$ in the following form
$$dY_t = \kappa Y_t \, dt + \sqrt{2} d B^1_t  -
\frac{Y_t}{\E(\|Y_t\|^2)} \E(Y_t \cdot \kappa Y_t) \, dt - d
\frac{Y_t}{\E(\|Y_t\|^2)} dt.$$
This agrees with~\eqref{eq:Xt_modif} for $(R=\Id,\lambda=d)$, thereby
providing a justification of our particular choice~\eqref{eq:choix1} in the previous
section.

To summarize, the original model for rigid rods~\eqref{eq:Xt} may be seen as a projection of the dynamics~\eqref{eq:original-dynamics} onto the submanifold $\|X_t\|^2=\|X_0\|^2$, while the approximated model~\eqref{eq:Xt_modif} for $(R=\Id,\lambda=d)$ which is consistent with the Doi closure~\eqref{eq:doi_closure} may be seen as  the
original dynamics~\eqref{eq:original-dynamics} constrained to have a
fixed {\em average} length: $\E \left(\|X_t\|^2\right)=\E\left(\|X_0\|^2\right)$. This yiels a microscopic interpretation of the Doi closure.

\section{Long-time behaviour of our approximate model}
\label{sec:convergence}
We henceforth consider the model~\eqref{eq:Xt_modif}--\eqref{eq:choix1} (namely $(R=\Id,\lambda=d)$) which we have built from the original
model~\eqref{eq:Xt}  by approximation. Throughout this section, we work
in dimension $d=2$. This is a crucial assumption, specifically needed for
our technique of proof which makes use of the Poincar\'e-Bendixson
Theorem. Additionally, we assume that the matrix $\kappa$ is defined by~\eqref{eq:kappa}:
\begin{equation*}
  \kappa=\frac{\Pe}{2}
\left[
\begin{array}{cc}
0 & a+1 \\
a-1 & 0
\end{array}
\right]
\end{equation*} 
 and that the initial condition $X_0$ satisfies $\E(\|X_0\|^2)=L^2=1$
 (so that, for all positive time, $\E(\|X_t\|^2)=1$). Given these
 assumptions,  we now recall, for the convenience of our reader and
the consistency of the present section, the model under study:  
\begin{equation}\label{Eq:Xt}
\begin{aligned}
dX_t &= \left( \kappa X_t  -  \kappa: \E( X_t \otimes X_t)  \, X_t \right)\, dt \\
& \quad   + 4N \left( \E(X_t \otimes X_t) X_t - \E( X_t \otimes X_t ): \E(X_t \otimes X_t) \, X_t   \right) \, dt \\
& \quad + \sqrt{2} \, dB_t - 2\, X_t \, dt.
\end{aligned}
\end{equation}
We also recall that the conformation tensor $M(t)=\E(X_t \otimes X_t)$ then satisfies the ordinary differential equation:
\begin{equation}\label{Eq:M}
\begin{aligned}
\frac{d M}{dt}&= \left( \kappa M + M \kappa^T - 2
\kappa : M  \, M \right)\\
&\quad+ 8 N \left( M^2 - M:M \, M \right) \\
& \quad +  4 (\Id/2 -  M),
\end{aligned}
\end{equation}
where $\tr(M(t))=\tr(M(0))=1$. The (non-linear) Fokker-Planck
formulation associated to~\eqref{Eq:Xt} and established in~\eqref{eq:FP1} writes:
\begin{equation}\label{Eq:FP}
\begin{aligned}
\frac{\partial \psi}{\partial t} &=
\div \left( \left( - \kappa x + \kappa:M[\psi] x \right) \psi \right)\\
 &\quad+ 4N
\div \left( \left( - M[\psi] x + M[\psi]:M[\psi] x \right) \psi \right)\\
& \quad + \Delta \psi + 2 \,\div \left( x \psi \right),
\end{aligned}
\end{equation}
where
$$M[\psi(t,\cdot)]=\int_{\R^2} x \otimes x \, \psi(t,x) \, dx$$ and $\tr(M[\psi(0,\cdot)])=1$. Notice that $t \mapsto M[\psi(t,\cdot)]$ then satisfies~\eqref{Eq:M}. 
The aim of this section is to study the longtime behaviour of the
solution $\psi$ to the Fokker-Planck equation~\eqref{Eq:FP}.

As it is standard for such analysis, we study the longtime behaviour of a  solution to the Fokker-Planck equation~\eqref{Eq:FP}, assumed sufficiently regular for our manipulations to be valid. We refer for example to~\cite{arnold-markowich-toscani-unterreiter-01} for an appropriate functional setting to justify such calculations.

\subsection{Long-time convergence of the solution to
  \eqref{Eq:M} to a periodic solution}
\label{ssec:existence}

We first consider the closed ordinary differential equation~\eqref{Eq:M}
on~$M$, momentarily leaving~\eqref{Eq:FP} aside. We will show that, under some assumptions on the parameters and the initial condition $M(0)$, $M$ converges, in the long time limit, to a periodic solution. This is an extension of the result~\cite[Theorem 5.1]{lee-forest-zhou-06}. We also refer to that contribution for a more thorough study of the longtime behaviour of the dynamical system, for other regimes of the parameters.
\begin{prop}
\label{prop:exist-periodic}
 Assume that $\Pe$ is sufficiently small, $|a|<1$ and
 $N>\frac{1}{1-a^2}$. Then there exists an open subset $\Omega$ of the
 ensemble of  positive definite matrices with trace one ($\Omega$ will be made precise in the course of the proof below), such that:
\begin{itemize}
\item There exists a unique periodic-in-time
  function~ $M_{per}(t)$, valued in $\Omega$, solution to~\eqref{Eq:M};
\item For any initial condition $M(0) \in \Omega$, the solution $M(t)$
  to~\eqref{Eq:M} converges to $M_{per}(t)$ exponentially fast in the
  long time, that is: there exist $C, \lambda >0$, such that, for all $t \ge 0$,
\begin{equation}\label{eq:exp_CV} 
\|M(t)-M_{per}(t) \| \le C \exp (-\lambda t).
\end{equation}
\end{itemize}
\end{prop}
The remainder of this section is devoted to the proof of
Proposition~\ref{prop:exist-periodic}.

\begin{proof}

\noindent{\it Step 1: Existence of a periodic-in-time solution}

\noindent  Using Proposition~\ref{prop:sym/trace}, we know that the solution $M(t)$ to
 \eqref{eq:M} is symmetric and satisfies
 $\tr(M(t))=1$ since this holds true at initial time. Introducing as in Section~\ref{sec:forest} the
 traceless part $Q=M-\Id / 2$ of~$M$, we may always write $Q$, in the two
 dimensional setting we consider, under the form
\begin{eqnarray*}  
Q(t)=\left[ \begin{array}{cc}
x(t)&y(t)\\
y(t)&-x(t)\\
\end{array}\right].
      \end{eqnarray*}
 Now the evolution equation \eqref{eq:Q} equivalently reads
 \begin{equation}\label{o1}  
\left\{ 
\begin{aligned} 
\f{dx}{dt}&=-4x \left( 1- N + 4N(x^2+y^2) \right) + \Pe \, y (1 -2ax),\\
 \f{dy}{dt}&=-4y \left( 1- N + 4N(x^2+y^2)\right)+\Pe \left(-x+\f{a}{2}-2ay^2\right). \end{aligned}\right.
 \end{equation}
Introducing the polar coordinates
\begin{eqnarray}
\label{eq:polar}
  \left\{ \begin{aligned} &x=r \cos \varphi,\\
 &y=r \sin \varphi, \end{aligned}\right. \end{eqnarray}
we rewrite~\eqref{o1} in the form 
\begin{equation}\label{o2}  
\left\{ 
\begin{aligned}
\f{d\varphi}{dt}&=- \Pe\left(1-\f{a}{2r}\cos \varphi\right),\\
 \f{dr}{dt}&=- 4 r  \left(  N(4 r^2-1)+1 \right)+ \frac{a \Pe}{2}(1-4 r^2)\sin \varphi.  
\end{aligned}
\right.
 \end{equation}

We now consider two positive constants $\epsilon_1$ and
$\epsilon_2$  satisfying $0<\epsilon_1<\frac{N-1}{4N}$ and
$0<\epsilon_2<\frac{1}{4N}$ respectively. Set
$r_1=\sqrt{\frac{N-1}{4N}-\epsilon_1}$ and
$r_2=\sqrt{\frac{N-1}{4N}+\epsilon_2}$. Note that, by construction,
$0<r_1\leq r_2<\frac{1}{2}$. Then, if $r_1\le r(0) \le r_2$, one has
\begin{equation*}
r_1\le r(t)\le r_2
\end{equation*}
for all positive times, provided that $\Pe$ is sufficiently small.
Indeed, it is easy to check that if $\Pe < \overline{\Pe}$, where 
$$\overline{\Pe}= \frac{16   N  \epsilon_1} { |a| }\,\min\left(
  \frac{r_1} {\frac{1}{N}+4\epsilon_1}, \frac{r_2}{
    \frac{1}{N}-4\epsilon_2} \right) >0,$$ then for all angles~$\varphi$,
\begin{equation*}
- 4 r_1  \left(  N(4 r_1^2-1)+1 \right)+ \frac{a \Pe}{2}(1-4 r_1^2)\sin \varphi = 
16 r_1  N  \epsilon_1+ \frac{a \Pe}{2}  \left(\frac{1}{N}+4\epsilon_1\right)\sin \varphi 
>0
\end{equation*}
and likewise,
\begin{equation*}
- 4 r_2  \left(  N(4 r_2^2-1)+1 \right)+ \frac{a \Pe}{2}(1-4 r_2^2)\sin \varphi = 
- 16 r_2  N  \epsilon_2+ \frac{a \Pe}{2}  \left(\frac{1}{N}-4\epsilon_2\right)\sin \varphi <0.
\end{equation*}
This shows that $\displaystyle \frac{dr}{dt}\biggl|_{r=r_1}>0$ and
$\displaystyle \frac{dr}{dt}\biggl|_{r=r_2}<0$, thus the
annulus $$\widetilde{\Omega}=\left\{(x,y), r_1 < \sqrt{x^2 +y^2} <
  r_2\right\}$$ is stable under the flow (for positive time). The domain
$\Omega$ mentioned above in the Proposition~\ref{prop:exist-periodic} is
now made precise and defined by:
$$(x,y) \in \widetilde{\Omega} \iff M=\left[ \begin{array}{cc}
x&y\\
y&-x\\
\end{array}\right] + \frac{\Id}{2} \in \Omega.$$

We now want to show that there is no stationary point in the annulus
$\widetilde{\Omega}$. Since, by assumption, $N > \frac{1}{1-a^2}$, we may assume that $\epsilon_1>0$ is chosen sufficiently small so that 
$$N > \frac{1}{1-4 \epsilon_1 -a^2}$$
which is equivalent to $\displaystyle\frac{|a|}{2r} < 1$ in $\widetilde{\Omega}$. This implies that \begin{equation}\label{eq:phidotneg}
\displaystyle\frac{d \varphi }{dt} = - \Pe\left(1-\f{a}{2r}\cos \varphi\right) <0
\end{equation}
since the solution remains in $\widetilde{\Omega}$, and thus that there is no stationary point in $\widetilde{\Omega}$.

From the Poincar\'e-Bendixson Theorem (see for example~\cite[Theorem
6.12]{meiss-07}), we then obtain that, for any trajectory with initial condition $(x_0,y_0)$ in
  $\widetilde{\Omega}$, its $\omega$-limit set is a periodic orbit, that
  is, the trajectory of a periodic solution. We recall, for consistency, that a point $(x_\infty,y_\infty)$ is in the $\omega$-limit set of the (forward) trajectory starting from $(x_0,y_0)$ if there is an increasing sequence of times $t_n$ going to infinity and such that $(x(t_n),y(t_n))$ converges to $(x_\infty,y_\infty)$ as $n$ goes to infinity.

\noindent A corollary of the previous statement is that there exists a periodic solution $(x_{per},y_{per})(t)$ to the equation~\eqref{o1} in $\widetilde{\Omega}$, and thus an associated periodic solution $M_{per}(t)$ to~\eqref{Eq:M},

\medskip

\noindent{\it Step 2: Properties of the periodic-in-time solution}

\noindent In order to prove the long time convergence to the
periodic solution and make precise the rate of that convergence, we now compute the divergence of the vector field in the right-hand-side of~\eqref{o1}: for  $(x,y) \in \widetilde{\Omega}$,
\begin{align*}
D(x,y)&=\frac{\partial}{\partial x} \left( -4x \left( 1- N + 4N(x^2+y^2) \right) + \Pe \, y (1 -2ax) \right) \\
& \quad+ \frac{\partial}{\partial y} \left(-4y \left( 1- N + 4N(x^2+y^2)\right)+\Pe \left(-x+\f{a}{2}-2ay^2\right)\right) \\
&= 8  \left( N - 1 \right)   - 64 N (x^2 + y^2) - 6 \Pe \,  a y\\
& \le 8  \left( N - 1 \right) - 64 N r_1^2 + 6 \Pe \, a \, r_2 \\
& = - 8  \left( N - 1 \right)  + 64 N\epsilon_1  + 6 \Pe \, a r_2.
\end{align*}
We thus see that if $\epsilon_1$ and $\Pe$ are chosen sufficiently small, then
$$D(x,y) < 0$$
in $\widetilde{\Omega}$. Following~\cite{kuznetsov-98,meiss-07}, we deduce that the equation~\eqref{o1} has a {\em unique stable} periodic orbit in $\widetilde{\Omega}$. 
The uniqueness follows from a generalization of the Dulac criterion. Let us recall the main arguments for the stability
statement. We introduce the Poincar\'e map associated to the first return to the section
$$S=\{(x,y) \in \widetilde{\Omega} \text{ such that } x>0, y=0 \}$$
of $\widetilde{\Omega}$.
It is a standard result~\cite[Equation (1.17)]{kuznetsov-98} or~\cite[Equation (4.51)]{meiss-07} that
\begin{equation}\label{eq:rho}
\rho=\exp \left(\int_0^T D(x_{per}(t),y_{per}(t)) \, dt\right),
\end{equation}
where $T$ is the period of the periodic solution $(x_{per},y_{per})(t)$,
gives the derivative of the Poincar\'e map at its stationary point (namely
the point where the periodic orbit intersects $S$). 
Thus,  $D <0$ implies that the derivative is strictly smaller than one ($\rho < 1$), which yields the exponential convergence to the stationary point of the Poincar\'e map, and thus the (exponential) asymptotic stability of the periodic orbit.

For further use (see the proof of Proposition~\ref{prop:ISL})), we
establish uniform bounds on the eigenvalues of  $M_{per}$. We first note
that $M_{per}$ is a symmetric {\em positive} matrix. It  is a simple consequence of the fact that $(x_{per}(t),y_{per}(t)) \in \widetilde{\Omega}$. Indeed, one can check that $\tr(M_{per}(t))=1$ and $$\det( M_{per}(t))=\frac{1}{4}-(x_{per}^2(t)+y_{per}^2(t)) \ge \frac{1}{4} - r_2^2 > 0.$$
This implies that the eigenvalues of $M_{per}(t)$ are $\frac{1}{2} \pm
\sqrt{x_{per}^2(t)+y_{per}^2(t)}$, and thus bounded both from below and
from above. In the sense of symmetric matrices, 
\begin{equation}\label{eq:M_per}
0< \left( \frac{1}{2} - r_2 \right) \Id \le M_{per}(t) \le \left(
  \frac{1}{2} + r_2 \right) \Id.
\end{equation}

\medskip

\noindent{\it Step 3: Convergence in the long time}

\noindent  To conclude our proof, we now show the
convergence~\eqref{eq:exp_CV}. Consider a solution~$(x,y)(t)$
to~\eqref{o1} and the sequence $$(x,y)(t_k)=(x_k,0)$$ of its successive
return points to the section $S$. Otherwise stated,  $x_{k+1}$ is the
image of $x_k$ by the Poincar\'e map. Notice that $x_k$ is a monotonic
sequence (since two trajectories cannot cross). As explained in Step 2, we also know that the
sequence $x_k$ converges exponentially fast to the fixed point $x^*$ of the Poincar\'e map: there exists $C>0$ and $\tilde{\rho} \in (0,1)$ (which can be chosen arbitrarily close to $\rho$) such that, for all $k \ge 0$,
\begin{equation}\label{eq:ykystar}
|x_k - x^*| \le C \tilde{\rho}^k.
\end{equation}
%
%
Since $(x^*,0)$ is on the periodic orbit, there exists a time $t^* \in [0,T)$ such that
$$(x_{per},y_{per})(t^*) =(x^*,0).$$
Without loss of generality, we may assume $t^*=0$. We now remark,
using~\eqref{eq:phidotneg}, that $\varphi:[t_k,t_{k+1}) \mapsto [0,2
\pi)$ is a one-to-one (actually strictly decreasing) function. We may
therefore use $\varphi$ itself  to reparameterize the trajectory between two successive return
points. It follows from~\eqref{o2} that
$$t_{k+1}-t_k= \frac{1 }{\Pe}\,\int_0^{2\pi} \frac{1 }{1 -
  \frac{a}{2r(\varphi^{-1}(\theta))} \cos\theta } d \theta.$$
 Likewise, along the periodic trajectory $(r_{per},\varphi_{per})(t)$ (namely $(x_{per},y_{per})(t)$ in polar coordinate), we have:
$$T=\frac{1 }{\Pe}\, \int_0^{2\pi} \frac{1 }{1 - \frac{a}{2r_{per}(\varphi_{per}^{-1}(\theta))} \cos \theta} d \theta.$$
We now use the fact that $(r,\varphi)(t_k)=(x_k,0)$ is close to
$(r_{per},\varphi_{per})(0)=(x^*,0)$ (by virtue of~\eqref{eq:ykystar})
and the Lipschitz property of the flow associated to~\eqref{o2} as a
function of  the initial conditions over a finite time interval,  to get that
$$|t_{k+1}-t_k - T| \le C \tilde{\rho}^k$$
where here and below $C>0$  denotes  irrelevant constants.

This implies that there exists a time $T_0$ such that for all $k \ge 0$,
$$|t_k - T_0 - k T| \le C {\tilde\rho}^{k}.$$

Indeed, $
t_k-t_0 
= \sum_{l=1}^k (t_l-t_{l-1})= kT+ \sum_{l=1}^k (t_l-t_{l-1} - T)
= kT+ \sum_{l\ge 1} (t_l-t_{l-1} - T) - \sum_{l>k} (t_l-t_{l-1} - T)$. 
Denoting $T_0=t_0+\sum_{l\ge 1} (t_l-t_{l-1} - T)$, we have
$t_k - T_0 - kT =- \sum_{l>k} (t_l-t_{l-1} - T)$, thus our claim.
%

Then, we have: for $k \ge 0$,
\begin{align*}
\|(&x_{per},y_{per}) (k T) - (x,y)(T_0 + k T)\|\\
& \le \|(x_{per},y_{per}) (k T) - (x,y)(t_k)\| + \|(x,y) (t_k) - (x,y)(T_0 + kT)\| \\
& \le |x^* - x_k| + C \left| t_k - ( T_0 + k T ) \right| \\
& \le C {\tilde \rho}^k.
\end{align*}
Using the Lipschitz property of the flow, we thus get: for all $k \ge 1$ and for all $t \in [kT,(k+1)T)$
$$\|(x_{per},y_{per}) (t) - (x,y)(T_0 + t)\| \le C {\tilde \rho}^k.$$
This implies~\eqref{eq:exp_CV} with $\lambda= - \ln ({\tilde \rho})$.
\end{proof}

\subsection{Analysis of  the Fokker-Planck equation~\eqref{Eq:FP}  for
  $M(t)$ given}

In this section, we consider~\eqref{Eq:FP} for a given $M(t)$: \begin{equation}\label{eq:FP_lin}
\begin{aligned}
\frac{\partial \psi}{\partial t} &=
\div \left( \left( - \kappa x + \kappa:M x \right) \psi \right)\\
 &\quad+ 4N
\div \left( \left( - M x + M:M x \right) \psi \right)\\
& \quad + \Delta \psi + 2 \,\div \left( x \psi \right).
\end{aligned}
\end{equation}
This equation can be rewritten in the form
\begin{equation}\label{eq:FPK}
{{\partial \psi}\over{\partial t}} = \div( K(t) x \psi) + \Delta \psi
\end{equation}
with
\begin{equation}\label{eq:K}
K= - \kappa + \kappa:M \Id + 4N \left(  - M + M:M \Id \right) + 2\Id.
\end{equation}
We are thus considering here a {\em linear} Fokker-Planck equation.

We begin with a basic remark, see for example~\cite{dolbeault-kinderlehrer-kowalczyk-02,bartier-dolbeault-illner-kowalczyk-07}.
\begin{prop}
\label{prop:entropy}
Let $\psi_1$ and  $\psi_2$ be two solutions to the Fokker-Planck
equation~\eqref{eq:FPK}. Then
$$\frac{d}{dt} H(\psi_1 | \psi_2) = - I (\psi_1|\psi_2)$$
where $$H(\psi_1 | \psi_2) = \int \ln (\psi_1 / \psi_2) \psi_1$$
denotes the {\em relative entropy} (of $\psi_1$ with respect to $\psi_2$), 
and  $$I(\psi_1 | \psi_2) = \int |\nabla \ln (\psi_1 / \psi_2)|^2
\psi_1$$
is the \emph{Fisher information} (of $\psi_1$ with respect to $\psi_2$).
\end{prop}
\begin{proof}
Set $b(t,x)= K(t)x$. We argue formally. Our manipulations are standard
and can be
made rigorous using appropriate functional spaces and cut-off functions. We write
\begin{align*}
\frac{d}{dt} H(\psi_1 | \psi_2)
&= \int \partial_t \psi_1  - \int \partial_t \psi_2 \frac{\psi_1}{\psi_2} + \int \ln\left(\frac{\psi_1}{\psi_2} \right) \partial_t \psi_1 \\
&= 0 - \int \div( b \psi_2 + \nabla \psi_2) \frac{\psi_1}{\psi_2}  + \int \div( b \psi_1 + \nabla \psi_1) \ln\left(\frac{\psi_1}{\psi_2} \right)\\
&= \int ( b \psi_2 + \nabla \psi_2 ) \cdot \nabla \frac{\psi_1}{\psi_2}  - \int ( b \psi_1 + \nabla \psi_1 ) \cdot \nabla \ln\left(\frac{\psi_1}{\psi_2} \right)\\
&= \int \nabla \psi_2  \cdot \nabla \frac{\psi_1}{\psi_2}  - \int  \nabla \psi_1  \cdot \nabla \ln\left(\frac{\psi_1}{\psi_2} \right)\\
&= \int\left( \frac{\psi_1}{\psi_2} \nabla \psi_2    -   \nabla \psi_1 \right) \cdot \nabla \ln\left(\frac{\psi_1}{\psi_2} \right)\\
&= - I(\psi_1 | \psi_2).
\end{align*}
Notice that this proof does not require $M$ to satisfy~\eqref{Eq:M}.
\end{proof}

We now build an explicit Gaussian solution to~\eqref{eq:FP_lin}.
\begin{prop}\label{prop:gauss}
Let $M(t)$ be a given time-dependent symmetric definite positive matrix with $\tr(M(0))=1$. Introduce the associated two-dimensional Gaussian probability density function
$$\psi_M(t,x)=\f{\sqrt{\det M^{-1}(t)}}{2\pi}\exp\left(-\f{x^TM^{-1}(t)\,x}2\right).$$
Then, $\psi_M$ satisfies~\eqref{eq:FP_lin} (for the given function $M(t)$) if and only if $M(t)$ satisfies~\eqref{Eq:M}.
\end{prop}
\begin{proof}
Let us denote $P(t)=M^{-1}(t)$, so that
\begin{equation*}
\psi_M(t,x)=\f{\sqrt{\det P(t)}}{2\pi}\exp\left(-\f{x^TP(t)x}2\right).
\end{equation*}
It is easy to calculate that
\begin{eqnarray*}  \partial_t \psi_M=
  -\f{x^T \partial_tPx}2\psi_M+\f{\partial_t \left(\sqrt{\det
      P}\right)}{2\pi}\exp\left(-\f{x^TPx}2\right),\end{eqnarray*}
where here and for the rest of this
proof we use the short-hand notation $\partial_t$ instead of
${{\partial }\over{\partial t}}$. 
Note that (using the symmetry of $P$),
\begin{align*}
\partial_t \left(\sqrt{\det P}\right)
&=\f12 \left(\sqrt{\det P}\right)^{-1}\partial_t(\det P)\\
&=\f12\sqrt{\det P}\,\partial_t(\ln \det P)=\f12\sqrt{\det P}\,\tr(P^{-1}\partial_tP).
\end{align*}
It follows that
\begin{equation*}
\partial_t \psi_M=\left( -\f{x^T(\partial_tP)x}2+\f12\tr(P^{-1}\partial_tP) \right)\psi_{M}. 
\end{equation*}
Plugging $\psi_M$ in the right hand side of the Fokker-Planck equation~\eqref{Eq:FP} yields 
\begin{eqnarray*} \div (Kx\psi_M+\nabla \psi_M)=\tr (K-P)\psi_M-x^TP(K-P)x\psi_M, \end{eqnarray*}
where $K$ is defined by~\eqref{eq:K}, so that $\psi_M$ is solution to~\eqref{Eq:FP} if and only if
\begin{equation*}
  -\f{x^T(\partial_tP)x}2+\f12\tr(P^{-1}\partial_tP)=\tr
  (K-P) -x^TP(K-P)x,  
\end{equation*}
for all $x\in\R^2$. This is equivalent to the couple of conditions
\begin{equation}
\label{eq:couple} 
\left\{
\begin{aligned}
\f12\tr(P^{-1}\partial_tP)&=\tr (K-P), \\
    \partial_t P&=PK+K^TP-2P^2,
  \end{aligned}
\right.
\end{equation}
where, for the second line, we have equated the \emph{symmetric} part of
the two second order coefficients. We immediately remark that the second
line of~\eqref{eq:couple} 
implies the first line, by elementary properties of the trace.

We now write \eqref{Eq:M}  under the form 
\begin{equation*}
\partial_t M=-(KM+MK^T)+2\Id,
\end{equation*}
where, again, $K$ is defined by~\eqref{eq:K}. Thus $M$ satisfies~\eqref{Eq:M} if and only if $P=M^{-1}$ satisfies
\begin{equation}\label{eq:couple2} 
\partial_t P= PK+K^TP-2P^2,
\end{equation}
where we use the fact $\partial_t P=-P \partial_t M P$.

By comparing~\eqref{eq:couple} and~\eqref{eq:couple2}, we thus get the result: $\psi_M$ satisfies~\eqref{Eq:FP} if and only if $M(t)$ satisfies~\eqref{Eq:M}.
\end{proof}

\begin{remark}
Notice that this proposition implies the existence of Gaussian solutions to the {\em non-linear} Fokker-Planck equation~\eqref{Eq:FP}. Indeed, let us consider a Gaussian initial condition $\psi(0,\cdot)$ and the associated $\displaystyle M(0)=\int x \otimes x \, \psi(0,x) \, dx$ initial condition to~\eqref{Eq:M}. Let $M(t)$ be the solution to~\eqref{Eq:M} with this initial condition. Let us then consider $\psi_M$ the Gaussian solution to~\eqref{eq:FP_lin} built in the previous proposition. We notice that $\psi_M$ is then a solution to the non-linear Fokker-Planck equation~\eqref{Eq:FP}, by uniqueness of the solution to~\eqref{Eq:M}.
\end{remark}
 
We now proceed with a uniqueness result for the periodic
solution to~\eqref{Eq:FP}. This result, for which we provide here a self
contained proof,  is also a consequence of the
convergence stated in Proposition~\ref{prop:CV} and proved in the next section.
\begin{prop}\label{prop:uniq}
Consider $\psi_{per}(t,x)$ a (sufficiently regular) periodic solution to the non-linear Fokker-Planck equation~\eqref{Eq:FP}. Define the associated periodic time-dependent matrix $$M_{per}(t)=M[\psi_{per}(t,\cdot)]=\int_{\R^2} x \otimes x \, \psi_{per}(t,x) \, dx.$$
Assume that $M_{per}(0) \in \Omega$, 
where $\Omega$ has been introduced in Proposition~\ref{prop:exist-periodic}. Then, $M_{per}$ is the unique periodic solution with value in $\Omega$ built in Proposition~\ref{prop:exist-periodic} and 
$$\psi_{per}=\psi_{M_{per}}$$
is the associated Gaussian solution built in Proposition~\ref{prop:gauss}.
\end{prop}
\begin{proof}
The uniqueness of the periodic orbit in $\Omega$ ensures that $M_{per}$ is indeed the periodic solution built in Proposition~\ref{prop:exist-periodic} (up to a phase change we may ignore).

Then, by uniqueness of solutions to~\eqref{Eq:M}, we notice that
$\psi_{per}$ and $\psi_{M_{per}}$ are both solutions to the same {\em
  linear} Fokker-Planck equation~\eqref{eq:FP_lin}, with
$M(t)=M_{per}(t)$. The question is thus now the following: given the
function $M_{per}$, we have to show that the periodic
solution to \eqref{eq:FP_lin} is unique (up to a normalization factor of course) and equal to~$\psi_{M_{per}}$.

Let us denote $\psi_{per}$ a periodic solution. We want to show that $\psi_{per}=\psi_{M_{per}}$. Arguing as in Proposition~\ref{prop:entropy}, we have
\begin{equation}\label{eq:ent_per} 
\f{d}{dt}
H(\psi_{per}| \psi_{M_{per}})+I(\psi_{per}| \psi_{M_{per}})=0.
\end{equation}
If $\psi_{per}$ and $\psi_{M_{per}}$ share the same period, say $T$,
then the function $H(\psi_{per}|\psi_{M_{per}})$ is also $T$-periodic
in time, and thus integrating the above equation from 0 to $T$, we
obtain 
\begin{eqnarray*} \int_0^T
  I(\psi_{per}|\psi_{M_{per}})dt=0. \end{eqnarray*}
Since the function $I$ is nonnegative, this immediately implies that
$\psi_{per}=\psi_{M_{per}}$ on $[0,T]$ and thus for any time.

If the period $\tilde T$ of  $\psi_{per}$  is different from the period
$T$ of $\psi_{M_{per}}$, we slightly adapt the above argument. From standard results on continuous fractions (see for example~\cite{hardy-wright-79,niven-zuckerman-montgomery-91}), there exists sequences of integers $p_n$, $q_n$ such that
$$\left|q_n \frac{\tilde{T}}{T} - p_n\right| < \frac{1}{q_{n+1}} \le \frac{2}{\phi^{n+1}}$$
where $\phi=\frac{1+\sqrt{5}}{2}$. Thus, if we set $\tau_n=p_n T$, we have $\tau_n = q_n \tilde{T} + \varepsilon_n$ where $\lim_{n \to \infty}\varepsilon_n = 0$ and $\lim_{n \to \infty}\tau_n = \infty$. Then, we have, for $n$ sufficiently large, 
\begin{align*}
&\left|H(\psi_{per}|\psi_{M_{per}})\, (\tau_n) -H(\psi_{per}|\psi_{M_{per}})
\,(0)\right|\\
& \le
\left|H(\psi_{per}|\psi_{M_{per}})\, (\tau_n)  -
H(\psi_{per}(p_n T)|\psi_{M_{per}}(q_n \tilde{T}))  \right|\\
& \quad +
\left|H(\psi_{per}(p_n T)|\psi_{M_{per}}(q_n \tilde{T}))-H(\psi_{per}|\psi_{M_{per}})
\,(0)\right|\\
& =\left|H(\psi_{per}(p_n T) |\psi_{M_{per}}(q_n \tilde{T} + \varepsilon_n) )\,   -
H(\psi_{per}(p_n T)|\psi_{M_{per}}(q_n \tilde{T}))  \right|\\
&= \left|  \int \psi_{per}(0, \cdot) \ln \left( \frac{\psi_{M_{per}}(0, \cdot)}{\psi_{M_{per}}(\varepsilon_n, \cdot)} \right) \right|
\end{align*}
and the right-hand side converges to zero as $n$ goes to infinity, using the continuity in time of $M_{per}$, and the fact that $\int \psi_{per}(0,x) \|x\|^2 \, dx < \infty$.
%
%

Integrating~\eqref{eq:ent_per} in time from $0$ to $\tau_n$, we thus get
$$\lim_{n \to \infty} \int_0^{{\tau}_n}
  I(\psi_{per}|\psi_{M_{per}})\,dt=0,$$ and thus $$ \int_0^{\max(T,\tilde T)}
  I(\psi_{per}|\psi_{M_{per}})\,dt=0$$  since $I$ is
  nonnegative and ${\tau_n}\geq \max(T,\tilde T)$ for large enough $n$. This shows that
  $\psi_{per}=\psi_{M_{per}}$ on the
  time interval $(0,\max(T,\tilde T))$, and therefore for all times.
\end{proof}

We conclude this section with an inequality that will be useful below to show exponential convergence to periodic solutions for~\eqref{Eq:FP}.
\begin{prop}\label{prop:ISL}
Let $M_{per}$ be the periodic solution to~\eqref{Eq:M}, and $\psi_{per}=\psi_{M_{per}}$ the associated (Gaussian) periodic solution to~\eqref{Eq:FP}. Then, $\psi_{per}$ satisfies a uniform in time logarithmic Sobolev inequality in the following sense: there exists $\mu >0$ such that for any probability density function $\psi$ and for any time $t \in [0,T)$,
\begin{equation}\label{eq:mu} 
H(\psi|\psi_{per}(t,\cdot)) \le \frac{1}{2\mu} I(\psi|\psi_{per}(t,\cdot)).
\end{equation}
\end{prop} 
\begin{proof}
It is well known that centered Gaussian distributions with covariance
matrix $M$ satisfy a logarithmic Sobolev inequality with parameter the
inverse of the largest eigenvalue of $M$, see for example~\cite{ABC-00}.

Now, \eqref{eq:M_per} precisely shows that the eigenvalues of $M$
are  uniformly  bounded from above by a time-independent constant.
This concludes the proof.
\end{proof}

\subsection{Long time convergence of the solution to the non-linear Fokker-Planck equation~\eqref{Eq:FP} to a periodic solution}
\label{ssec:convergence}
Our final step is to 
address the long time behaviour of the solution to the non-linear Fokker-Planck
equation~\eqref{Eq:FP}. Assume the parameters $(a,N,\Pe)$ are such that
the conclusion of Proposition~\ref{prop:exist-periodic} holds:
$M(t)=M[\psi(t,\cdot)]$ converges exponentially fast to the periodic
solution $M_{per}(t)$ (with value in $\Omega$) of~\eqref{eq:M}. Let us
denote $$\psi_{per}=\psi_{M_{per}}$$ the associated unique periodic
solution to~\eqref{Eq:FP} (see Proposition~\ref{prop:uniq} above).
Then, consider a solution $\psi$ of~\eqref{Eq:FP} and assume that the initial condition $\psi(0)$ satisfies
\begin{equation}\label{eq:hyp}
\int_{\R^2} x \otimes x \, \psi(0,x) \, dx = M(0) \in \Omega,
\end{equation}
where $\Omega$ is defined above. We have:
\begin{prop}\label{prop:CV}
Under the assumptions of Proposition~\ref{prop:exist-periodic} (and in particular~\eqref{eq:hyp}), the solution $\psi$ to~\eqref{Eq:FP} (which we assume sufficiently regular) converges exponentially fast to $\psi_{per}$ in the following sense: there exist $C>0$ and $\nu > 0$ such that, for all time $t > 0$,
$$H(\psi(t,\cdot) | \psi_{per}(t, \cdot) ) \le C \exp (-\nu t).$$
\end{prop}
\begin{proof}
To the function $\psi$ is associated $M(t)=M[\psi(t,\cdot)]$ satisfying~\eqref{eq:M}, which (by  Proposition~\ref{prop:exist-periodic}) converges exponentially fast to $M_{per}$: $\forall t \ge  0$
\begin{equation}\label{eq:1}
\|M(t) - M_{per}(t) \| \le C \exp(-\lambda t).
\end{equation}
The function $\psi$ satisfies the linear Fokker-Planck equation:
\begin{equation}\label{eq:2}
\frac{\partial \psi}{\partial t}= \div (K(t)x \psi) + \Delta \psi,
\end{equation}
where $K$ is defined by~\eqref{eq:K}. Likewise, the function $\psi_{per}$ satisfies the linear Fokker-Planck equation:
$$\frac{\partial \psi}{\partial t}= \div (K_{per}(t)x \psi) + \Delta \psi,$$
where $K_{per}$ is the periodic function defined by~\eqref{eq:K} with $M=M_{per}$. Notice that from~\eqref{eq:1}, we get
\begin{equation}\label{eq:3}
\|K(t) - K_{per}(t) \| \le C \exp(-\lambda t).
\end{equation}

Now, adapting the proof of Proposition~\ref{prop:entropy} and rewriting~\eqref{eq:2} as:
$$\frac{\partial \psi}{\partial t}= \div (K_{per}(t)x \,\psi) + \Delta \psi + \div ([K(t)-K_{per}(t)]x \,\psi),$$
we have, for $0 < \varepsilon < 1$,
\begin{align*}
\frac{d}{dt} H(\psi|\psi_{per}) 
&= - I(\psi|\psi_{per}) + \int \div ([K(t)-K_{per}(t)]x \psi) \ln \left(\frac{\psi}{\psi_{per}} \right)\\
&=  - I(\psi|\psi_{per}) - \int [K(t)-K_{per}(t)]x \cdot \nabla \ln \left(\frac{\psi}{\psi_{per}}  \right) \psi \\
& \le - (1-\varepsilon) I(\psi|\psi_{per}) + \frac{1}{4 \varepsilon} \int \|[K(t)-K_{per}(t)]x\|^2 \psi \\
& \le - (1-\varepsilon) I(\psi|\psi_{per}) + \frac{1}{4 \varepsilon} \|K(t)-K_{per}(t)\|^2 \int \|x\|^2 \psi.
\end{align*}
Now, using~\eqref{eq:1}, the fact that $\int \|x\|^2 \psi=\tr(M)=1$, and Proposition~\ref{prop:ISL}, we get
\begin{align*}
\frac{d}{dt} H(\psi|\psi_{per}) 
& \le - \frac{(1-\varepsilon)}{2\mu} H(\psi|\psi_{per}) + \frac{C}{4 \varepsilon} \exp(-2\lambda t),
\end{align*}
from which we deduce the exponential convergence of $ H(\psi|\psi_{per}) $  to zero, using the Gronwall Lemma.
\end{proof}

\begin{remark}
It is easy, by making precise all the constants used in the bounds above, to give an expression for the rate of convergence $\nu$ in terms of $\rho$ defined in~\eqref{eq:rho} and $\mu$ defined in~\eqref{eq:mu}.
\end{remark}

\appendix

\section{Proof of Proposition~\ref{prop:doi}}

We provide in this appendix a proof of Proposition~\ref{prop:doi}.
Adapting a standard coupling approach, see for example~\cite{hitsuda-mitoma-86}, we introduce $N$ independent copies of the nonlinear stochastic differential equation~\eqref{eq:Y}:
$$dY^i_t = \kappa Y^i_t \, dt + \sqrt{2} d B^i_t  - \frac{Y^i_t}{L^2} \, \E(Y^i_t \cdot \kappa Y^i_t) \, dt - d \,\frac{Y^i_t}{L^2} dt,$$
driven by the \emph{same} Brownian motions as the processes $(X^{i,I}_t)_{t \ge 0}$. Let us recall that, from~\eqref{eq:IC1}, we have: for all positive time $t$,
\begin{equation}\label{eq:Xnorm}
\frac{1}{I} \sum_{i=1}^I \|X^{i,I}_t\|^2 = L^2,
\end{equation}
so that
$$\E\|X^{i,I}_t\|^2 = L^2$$
since the law of the stochastic process $\XX^I_t$ is invariant under permutation of the indices of its components $(X^{1,I}_t, \ldots, X^{I,I}_t)$.
Moreover, since $\E\|Y^i_0\|^2=L^2$, we also have, for all  positive time $t$,
\begin{equation}\label{eq:Ynorm}
\E(\|Y^i_t\|^2)=\E(\|Y_t\|^2)=L^2.
\end{equation}
This originates from the fact that $\E(\|Y_t\|^2)$ solves the ordinary
differential equation:\begin{align*}
\frac{d}{dt} \E(\|Y_t\|^2)
&= 2 \E(Y_t \cdot \kappa Y_t) \left( 1 - \frac{\E(\|Y_t\|^2)}{L^2}\right) + 2 d \left(1- \frac{\E(\|Y_t\|^2)}{L^2}\right).
\end{align*}
We now introduce the difference
$$Z^{i,I}_t=X^{i,I}_t- Y^i_t,$$
with initial condition $Z^{i,I}_0=Y^i_0\left(L\,\left(\frac{1}{ I} \sum_{i=1}^I \|Y^i_0 \|^2\right)^{-1/2}  - 1\right)$. We have:
\begin{align*}
Z^{i,I}_t &= Z^{i,I}_0 + \left(\kappa - \frac{d \Id}{L^2} \right) \int_0^t Z^{i,I}_s \, ds \\
& \quad - \frac{1}{L^2} \int_0^t \left( X^{i,I}_s\frac{1}{I} \sum_{j=1}^I X^{j,I}_s \cdot \kappa X^{j,I}_s - Y^i_s \, \E(Y^i_s \cdot \kappa Y^i_s) \right)\, ds \\
& \quad - \int_0^t \frac{X^{i,I}_s}{L} \frac{1}{\sqrt{I}} \frac{\XX^I_s \cdot d\BB_s}{\|\XX^I_s\|} + \frac{ 1}{IL^2} \int_0^t X^{i,I}_s ds.
\end{align*}
Thus, using the Doob inequality to get
\begin{align*}
\E \left( \sup_{s \le t} \int_0^s \| X^{i,I}_r \|  \frac{\XX^I_r \cdot d\BB_r}{\|\XX^I_r\|} \right)^2 
& \le C  \, \E \left(\int_0^t  \| X^{i,I}_r\|  \frac{\XX^I_r \cdot d\BB_r}{\|\XX^I_r\|} \right)^2 \\
&=  C  \,\int_0^t \E \|X^{i,I}_r\|^2 \, dr \le C L^2 T,
\end{align*}
(where here and throughout this proof, $C$ denotes irrelevant constants that
depend on $T$, $\kappa$, $d$, $L$ but not on $I$), we obtain, for any time $t \in [0,T]$, ($T$ is fixed),
\begin{align}
\E \left( \sup_{s \le t} \|Z^{i,I}_s\|^2  \right)
& \le \E \|Z^{i,I}_0\|^2 + C \int_0^t \E \left( \sup_{r \le s}  \|Z^{i,I}_r\|^2 \right) \, ds \nonumber\\
& \quad + C \int_0^t \E \left\| X^{i,I}_s\frac{1}{I} \sum_{j=1}^I X^{j,I}_s \cdot \kappa X^{j,I}_s - Y^i_s \, \E(Y^i_s \cdot \kappa Y^i_s) \right\|^2 \, ds \label{eq:diff} \\
& \quad + \frac{C}{I}. \nonumber
\end{align}
 We are now going to
estimate from above the first and third terms of the right-hand side.

We begin with the first term, which involves the initial condition $Z^{i,I}_0$. Using that, for all $x,y>0$, $(x-y)^2 \le |x^2 - y^2|$, we have:
\begin{align}
\E \|Z^{i,I}_0\|^2
&= \E \left(\|Y^i_0\|^2 \left(\frac{L}{\sqrt{\frac{1}{ I} \sum_{i=1}^I \|Y^i_0 \|^2 }}  - 1\right)^2 \right) \nonumber\\
&
= \E \left(\frac{1}{I} \sum_{j=1}^I\|Y^j_0\|^2 \left(\frac{L}{\sqrt{\frac{1}{ I} \sum_{i=1}^I \|Y^i_0 \|^2 }}  - 1\right)^2 \right)\nonumber\\
&= \E \left( \left(L - \sqrt{\frac{1}{ I} \sum_{i=1}^I \|Y^i_0 \|^2} \right)^2 \right)\nonumber\\
&= \E \left( \left( \frac{L^2 - \frac{1}{ I} \sum_{i=1}^I \|Y^i_0 \|^2}{L + \sqrt{\frac{1}{ I} \sum_{i=1}^I \|Y^i_0 \|^2}} \right)^2 \right)\nonumber\\
& \le \frac{1}{L^2} \, \E \left( \left( L^2 - \frac{1}{ I} \sum_{i=1}^I \|Y^i_0 \|^2 \right)^2 \right) = \frac{1}{I} \frac{\var(\|Y^i_0\|^2)}{L^2}. \label{eq:T0_control}
\end{align}

We now consider the term~\eqref{eq:diff}, which we split as follows:
\begin{align}
 X^{i,I}_s\frac{1}{I} \sum_{j=1}^I X^{j,I}_s \cdot \kappa X^{j,I}_s - Y^i_s \, \E(Y^i_s \cdot \kappa Y^i_s) 
&=( X^{i,I}_s - Y^i_s )\frac{1}{I} \sum_{j=1}^I X^{j,I}_s \cdot \kappa X^{j,I}_s \label{eq:T1} \\
&\quad + Y^{i}_s \frac{1}{I} \sum_{j=1}^I \left(X^{j,I}_s \cdot \kappa X^{j,I}_s - Y^{j}_s \cdot \kappa Y^{j}_s\right) \label{eq:T2}\\
& \quad +  Y^{i}_s \left( \frac{1}{I} \sum_{j=1}^I Y^{j}_s \cdot \kappa Y^{j}_s - \E(Y^i_s \cdot \kappa Y^i_s) \right).\label{eq:T3}
\end{align}
For the first term~\eqref{eq:T1}, we have, using~\eqref{eq:Xnorm}
\begin{equation}\label{eq:T1_control}
\E \left\| (X^{i,I}_s - Y^i_s) \frac{1}{I} \sum_{j=1}^I X^{j,I}_s \cdot \kappa X^{j,I}_s\right\|^2  \le C \, \E \| Z^{i,I}_s\|^2 \le C \, \E \left(\sup_{ r \le s} \| Z^{i,I}_r\|^2 \right).
\end{equation}
For the third term~\eqref{eq:T3}, we write:
\begin{align*}
&\E \left\| Y^{i}_s \left( \frac{1}{I} \sum_{j=1}^I Y^{j}_s \cdot \kappa Y^{j}_s - \E(Y^i_s \cdot \kappa Y^i_s) \right) \right\|^2\\
&=\frac{1}{I^2} \sum_{j,k=1}^I \E \left( \| Y^{i}_s \|^2   \left( Y^{j}_s \cdot \kappa Y^{j}_s - \E(Y^j_s \cdot \kappa Y^j_s) \right) \left( Y^{k}_s \cdot \kappa Y^{k}_s - \E(Y^k_s \cdot \kappa Y^k_s) \right)  \right).
\end{align*}
By independence of the stochastic processes $(Y^i_t)_{i \ge 1}$, the terms in the sum are zero if $j \neq k$. Thus,
\begin{eqnarray}
&&\E \left\| Y^{i}_s \left( \frac{1}{I} \sum_{j=1}^I Y^{j}_s \cdot \kappa Y^{j}_s - \E(Y^i_s \cdot \kappa Y^i_s) \right) \right\|^2 
=\frac{1}{I^2} \sum_{j=1}^I \E \left( \| Y^{i}_s \|^2   \left( Y^{j}_s \cdot \kappa Y^{j}_s - \E(Y^j_s \cdot \kappa Y^j_s) \right)^2 \right) \nonumber \\
&&\quad\quad\quad\quad \le \frac{1}{I^2} \sum_{j=1}^I \sqrt{ \E \left( \| Y^{i}_s \|^4\right)} \sqrt{ \E   \left( Y^{j}_s \cdot \kappa Y^{j}_s - \E(Y^j_s \cdot \kappa Y^j_s) \right)^4 }  \nonumber \\
&&\quad\quad\quad\quad\le \frac{C}{I} \label{eq:T3_control}
\end{eqnarray}
using that $\sup_{s \in [0,T]} \E \left( \| Y_s \|^8\right) < \infty$,
which is easy to check from~\eqref{eq:Y} provided the initial condition
$Y_0$ is assumed to have finite moments up to order 8. We finally estimate the second term~\eqref{eq:T2}. We have
\begin{align*}
&\E \left\| Y^{i}_s \frac{1}{I} \sum_{j=1}^I \left(X^{j,I}_s \cdot \kappa X^{j,I}_s - Y^{j}_s \cdot \kappa Y^{j}_s\right) \right\|^2\\
& \le C
\E \left( \|Y^{i}_s\|^2 \frac{1}{I^2} \left( \sum_{j=1}^I  \|Z^{j,I}_s\| \left( \|X^{j,I}_s \| +  \|Y^{j}_s\|\right) \right)^2 \right)\\
& \le C
\E \left( \|Y^{i}_s\|^2 \left( \frac{1}{I} \sum_{j=1}^I  \|Z^{j,I}_s\|^2 \right) \left( \frac{1}{I} \sum_{j=1}^I   \left( \|X^{j,I}_s \|^2 +  \|Y^{j}_s\|^2 \right)\right)  \right) \\
& = C
\E \left( \left( \frac{1}{I} \sum_{j=1}^I \|Y^{j}_s\|^2 \right) \left( \frac{1}{I} \sum_{j=1}^I  \|Z^{j,I}_s\|^2 \right) \left( L^2 + \frac{1}{I} \sum_{j=1}^I  \|Y^{j}_s\|^2 \right)  \right)\\
&= C L^2  \E \left( \left( \frac{1}{I} \sum_{j=1}^I \|Y^{j}_s\|^2 \right) \left( \frac{1}{I} \sum_{j=1}^I  \|Z^{j,I}_s\|^2 \right)   \right) + C \E \left( \left( \frac{1}{I} \sum_{j=1}^I \|Y^{j}_s\|^2 \right)^2 \left( \frac{1}{I} \sum_{j=1}^I  \|Z^{j,I}_s\|^2 \right) \right).
\end{align*}
Consider the second term (the first term is addressed similarly). We write
\begin{align}
\E &\left( \left( \frac{1}{I} \sum_{j=1}^I \|Y^{j}_s\|^2 \right)^2 \left( \frac{1}{I} \sum_{j=1}^I  \|Z^{j,I}_s\|^2 \right) \right)\nonumber \\
&=\E \left( \left( \frac{1}{I} \sum_{j=1}^I \|Y^{j}_s\|^2 \right)^2 \left( \frac{1}{I} \sum_{j=1}^I  \|Z^{j,I}_s\|^2 \right) 1_{\frac{1}{I} \sum_{j=1}^I \|Y^{j}_s\|^2 \le 2L^2} \right) \nonumber \\
& \quad +\E \left( \left( \frac{1}{I} \sum_{j=1}^I \|Y^{j}_s\|^2 \right)^2 \left( \frac{1}{I} \sum_{j=1}^I  \|Z^{j,I}_s\|^2 \right) 1_{\frac{1}{I} \sum_{j=1}^I \|Y^{j}_s\|^2 > 2L^2} \right)\nonumber \\
& \le 4 L^4 \E \|Z^{j,I}_s\|^2 + C \E \left( \left( \frac{1}{I} \sum_{j=1}^I \|Y^{j}_s\|^2 \right)^2 \left( \frac{1}{I} \sum_{j=1}^I  \left( \|X^{j,I}_s\|^2 +  \|Y^{j}_s\|^2\right) \right) 1_{\frac{1}{I} \sum_{j=1}^I \|Y^{j}_s\|^2 > 2L^2} \right)\nonumber \\
&= 4 L^4 \E \|Z^{j,I}_s\|^2 + C L^2 \E \left( \left( \frac{1}{I} \sum_{j=1}^I \|Y^{j}_s\|^2 \right)^2  1_{\frac{1}{I} \sum_{j=1}^I \|Y^{j}_s\|^2 > 2L^2} \right) \nonumber \\
& \quad + C \E \left( \left( \frac{1}{I} \sum_{j=1}^I \|Y^{j}_s\|^2 \right)^3 1_{\frac{1}{I} \sum_{j=1}^I \|Y^{j}_s\|^2 > 2L^2} \right)\nonumber \\
& \le C \E \left( \sup_{r \le s} \|Z^{j,I}_r\|^2 \right) + \frac{C}{I^2}. \label{eq:T2_control}
\end{align}
Here, we have used a concentration inequality, based on~\eqref{eq:Ynorm}: for $n=2,3$,
\begin{align*}
\E& \left( \left( \frac{1}{I} \sum_{j=1}^I \|Y^{j}_s\|^2 \right)^n 1_{\frac{1}{I} \sum_{j=1}^I \|Y^{j}_s\|^2 > 2L^2} \right)\\
&=\E \left( \left( \frac{1}{I} \sum_{j=1}^I \|Y^{j}_s\|^2 -L^2 + L^2 \right)^n 1_{\frac{1}{I} \sum_{j=1}^I \|Y^{j}_s\|^2 -L^2 > L^2} \right)\\
&\le 2^{n-1}\E \left( \left( \left(  \frac{1}{I} \sum_{j=1}^I \left( \|Y^{j}_s\|^2 -L^2\right) \right)^n + L^{2n} \right)  1_{\frac{1}{I} \sum_{j=1}^I \|Y^{j}_s\|^2 -L^2 > L^2} \right)\\
&=  2^{n-1}\E \left( \left( \frac{1}{I} \sum_{j=1}^I \left(
      \|Y^{j}_s\|^2 -L^2\right) \right)^n  1_{\frac{1}{I} \sum_{j=1}^I
    \|Y^{j}_s\|^2 -L^2 > L^2} \right)\\
&\quad\quad  + 2^{n-1} L^{2n} \E\left(1_{\frac{1}{I} \sum_{j=1}^I \|Y^{j}_s\|^2 -L^2 > L^2} \right) \\
& \le \frac{C}{I^2}.
\end{align*}
The last line is a consequence of the following: for $\xi_j=\|Y^{j}_s\|^2 -L^2$ i.i.d. centered random variables with finite fourth moment, $\alpha = L^2 > 0$, and $m=0,2,3$,
\begin{align*}
\E \left( \left( \frac{1}{I} \sum_{j=1}^I \xi_j \right)^m 1_{ \frac{1}{I} \sum_{j=1}^I \xi_j > \alpha} \right) \le \frac{1}{\alpha^{4-m}} \E\left( \left( \frac{1}{I} \sum_{j=1}^I \xi_j \right)^{4} \right) \le \frac{C}{I^2},
\end{align*}
where $C$ depends on $\sup_{s \in [0,T]} \E \left( \|Y^{j}_s\|^8 \right) < \infty$.

Inserting~\eqref{eq:T0_control}--\eqref{eq:T1_control}--\eqref{eq:T3_control}--\eqref{eq:T2_control} in~\eqref{eq:diff}, we obtain: for all time $t \in [0,T]$,
\begin{equation*}
\E \left( \sup_{s \le t} \|Z^{i,I}_s\|^2  \right)
 \le C \int_0^t \E \left( \sup_{r \le s}  \|Z^{i,I}_r\|^2 \right) \, ds + \frac{C}{I},
\end{equation*}
and we conclude using the Gronwall lemma.

\acknowledgement{The last two authors would like to thank Greg Forest for enlightening
  discussions, in particular during a
  workshop on polymer flows organized at Ecole des Ponts ParisTech in January 2009. The variety of observed and simulated  long
   term behaviours of nematic polymer flows along with the importance of
   mathematically understanding such phenomena were then pointed out. The help of Victor
   Kleptsyn regarding the theory of planar dynamical systems is also acknowledged.
}


\end{document}